\documentclass[reqno, 11pt]{amsart}
\usepackage[letterpaper,hmargin=1in,vmargin=1in]{geometry}

\usepackage{amsmath,amssymb,amsthm,mathrsfs,graphicx,url}
\usepackage[usenames,dvipsnames]{color}
\usepackage[ colorlinks=true,linkcolor=Red,citecolor=Green]{hyperref}
\usepackage{pinlabel}

\DeclareMathOperator \re {Re}

\DeclareMathOperator \supp {supp}
\newcommand{\Rb}{\mathbb{R}}
\newcommand{\Tb}{\mathbb{T}}
\newcommand{\Zb}{\mathbb{Z}}
\newcommand{\e}{\varepsilon}
\newcommand{\polydamp}{V}

\newtheorem{lem}{Lemma}
\newtheorem*{thm}{Theorem}

\author{Kiril Datchev}
\address{Department of Mathematics, Purdue University, West Lafayette, IN, USA}
\email{kdatchev@purdue.edu}

\author{Perry Kleinhenz}
\address{Department of Mathematics, Northwestern University, Evanston, IL, USA}
\email{pbk@math.northwestern.edu}

\title[Sharp polynomial decay rates]{Sharp polynomial decay rates for the damped wave equation with H\"older-like damping}

\date{\today}

\begin{document}

\begin{abstract}
We study decay rates for the energy of solutions of the damped wave equation on the torus. We consider dampings invariant in one direction and bounded above and below by multiples of $x^{\beta}$ near the boundary of the support and show decay at rate $1/t^{\frac{\beta+2}{\beta+3}}$. In the case where $W$ vanishes exactly like $x^{\beta}$ this result is optimal by \cite{Kleinhenz2019}. The proof uses a version of the Morawetz multiplier method.
\end{abstract}

\maketitle

\section{Introduction}

Let $W$ be a bounded, nonnegative damping function on a compact Riemannian manifold $M$, and let $v$ solve 
\begin{equation}\label{e:veq}\begin{cases}
 \partial_t^2 v + W \partial_t v - \Delta v = 0, \qquad &t>0, \\
(v,\partial_t v) = (v_0,v_1)\in C^\infty(M) \times C^\infty(M), \qquad &t=0.
\end{cases}\end{equation}
We are interested in decay rates as $t \to \infty$ for the energy
\[
 \mathcal E(t) = \int_{M} |\partial_t v(t)|^2 + |\nabla v(t)|^2.
\]
When $W$ is continuous, it is classical that \textit{uniform stabilization}, namely a uniform decay rate $\mathcal E(t) \le C r(t) \mathcal E(0)$ with $r(t) \to 0$ as $t \to \infty$, is equivalent to  \textit{geometric control}, namely the existence of a length $L$ such that all geodesics of length at least $L$ intersect the set where $W>0$. Moreover, in this case the optimal $r(t)$ is exponentially decaying in $t$.

When uniform stabilization fails, we look instead for $r(t)$ such that 
\begin{equation}\label{e:endecgen}
  \mathcal E(t)^{1/2} \le C r(t) \left(\|v_0\|_{H^2(M)} +  \|\partial_t v_1\|_{H^1(M)}  \right).
\end{equation}
Then the optimal $r(t)$ depends on the geometry of $M$ and of the set where $W>0$, and also on the rate of vanishing of  $W$. In this note we explore this dependence in precise detail for translation invariant damping functions on the torus, where we prove decay of the form
\begin{equation}\label{e:endec}
  \mathcal E(t)^{1/2} \le C t^{-\alpha} \left(\|v_0\|_{H^2(M)} +  \|\partial_t v_1\|_{H^1(M)}  \right).
\end{equation}

\begin{thm}
Let $M$ be the torus $(\Rb / 2 \pi  \Zb)_x \times (\Rb / 2 \pi  \Zb)_y$. Let $C_0>0$, $\sigma \in (0,\pi)$, and $\beta \geq 0$ be given. Suppose $W = W(x)$ obeys
\begin{equation}\label{e:wass}
\frac 1 {C_0} \polydamp(x) \le W(x) \le C_0 \polydamp(x), \qquad \polydamp(x) =
\begin{cases}
0, \qquad &|x| \in [0,\sigma],\\
(|x| - \sigma)^\beta, \qquad & |x| \in (\sigma, \pi],
\end{cases}
\end{equation}
for all $x \in [-\pi,\pi]$. Then there is $C$, depending only on $C_0$, $\sigma$, and $\beta$, such that \eqref{e:endec} holds with
\begin{equation}\label{e:ab23}
 \alpha = \frac{\beta+2}{\beta+3}.
\end{equation}

\end{thm}

\subsection*{Remarks}

\begin{enumerate}
\item Our result is especially interesting when $W=\polydamp$ near the set where $|x| = \sigma$. Then, by Theorem 1.1 of \cite{Kleinhenz2019}, the value of $\alpha$ in \eqref{e:ab23} is the best possible. More specifically, in this setting the second author proves that \eqref{e:endec} is false for any $\alpha > \frac{\beta+2}{\beta+3}$ by  constructing a suitable sequence of quasimodes of the stationary operator $-\Delta + i qW - q^2$.

\item As is clear from the reduction to \eqref{e:2b20} at the beginning of the proof below, the same proof gives the same result (with the same constant $C$) if the torus $\Tb =  (\Rb / 2 \pi \Zb)_x \times (\Rb / 2 \pi \Zb)_y$ is replaced by another product $(\Rb / 2 \pi \Zb)_x \times \Sigma_y$, where $\Sigma$ is any compact Riemannian manifold.


\end{enumerate}

The equivalence of uniform stabilization and geometric control for continuous damping functions was proved by Ralston \cite{Ralston1969}, Rauch and Taylor \cite{RauchTaylor1975} (see also \cite{BardosLebeauRauch1992} and \cite{BurqGerard1997}, where $M$ is also allowed to have a boundary). For some more recent finer results concerning discontinuous damping functions, see Burq and G\'erard \cite{BurqGerard2018}.

Decay rates of the form \eqref{e:endecgen} go back to Lebeau \cite{Lebeau1996}. If we assume only that $W \in C(M)$ is nonnegative and not identically $0$, then the best general result is that $r(t)$ in \eqref{e:endecgen} is $1/\log(2+t)$ \cite{Burq1998} and this is optimal on spheres and some other surfaces of revolution \cite{Lebeau1996}.  At the other extreme, if $M$ is a negatively curved (or Anosov) surface, $W \in C^\infty(M)$, $W \ge 0$, $W \not \equiv 0$, then $r(t)$ may be chosen exponentially decaying \cite{DyatlovJinNonnenmacher}. 

When $M$ is a torus, these extremes are avoided and the best bounds are polynomially decaying as in \eqref{e:endec}. Anantharaman and L\'eautaud \cite{AnantharamanLeautaud2014} show \eqref{e:endec} holds with $\alpha =1/2$ when $W \in L^\infty$, $W \ge 0$, $W>0$ on some open set, as a consequence of Schr\"odinger observability/control  \cite{Jaffard1990, Macia2010, BurqZworski2012}. The more recent result of Burq and Zworski \cite{BurqZworski2019} weakens the requirement that  $W>0$ on some open set to merely $W \not \equiv 0$. Anantharaman and L\'eautaud \cite{AnantharamanLeautaud2014} further show that if $\supp W$ does not satisfy the geometric control condition then \eqref{e:endec} cannot hold for any $\alpha>1$. They also show if $W$ satisfies $|\nabla W| \leq W^{1-\e}$ for $\e>0$ small enough and $W \in W^{k_0, \infty}$ for some $k_0$ then $\eqref{e:endec}$ holds with $\alpha =1/(1+4\e)$. 
For earlier work on the square and partially rectangular domains see \cite{LiuRao2005} and  \cite{BurqHitrik2007} respectively, and for  polynomial decay rates in the setting of a degenerately hyperbolic undamped set, see \cite{csvw}. 

In \cite{Kleinhenz2019}, the second author shows that, if $W=\polydamp$ near $\sigma$, then \eqref{e:endec} holds with  $\alpha=(\beta+2)/(\beta+4)$. In the case of constant damping on a strip ($W = \polydamp$ and $\beta = 0$) the result that \eqref{e:endec} holds with  $\alpha = 2/3$ is due to Stahn \cite{Stahn2017}, and the result that it does not hold for $\alpha>2/3$ is due to Nonnenmacher \cite{AnantharamanLeautaud2014}.

Our result holds for $\sigma \in (0,\pi)$, but one can also look at the behavior as $\sigma$ approaches the endpoints of the interval. As  $\sigma \to \pi$, the constants in our estimates blow up. This makes sense because the problem becomes undamped, and no decay is possible in the limit (\eqref{e:endec} holds only with $\alpha = 0$). More interesting is to let $\sigma \to 0$. In that case the constants in our estimates remain bounded, but better results are known by other methods.

In \cite{LeautaudLerner2017},  L\'eautaud and Lerner show that if $\sigma = 0$, then \eqref{e:endec} holds with $\alpha = (\beta+2)/\beta$. They also consider more general manifolds and damping functions. Note that, intriguingly, the decay rate decreases as $\beta$ increases when $\sigma = 0$, while  the decay rate increases as $\beta$ increases when $\sigma \in (0,\pi)$. A key difference in the geometry is that, when $\sigma = 0$ the support of $W$ is the whole torus (so all geodesics interesect it) whereas when $\sigma>0$ there is a one parameter family of geodesics in $(-\sigma, \sigma)$ which do not intersect the support of $W$.

Our result may be interpreted microlocally in the following way. The decay rates in \eqref{e:endecgen} and \eqref{e:endec} are related to time averages of $W$ along geodesics \cite{Nonnenmacher2011}. When $W$ has conormal singularities, as in the case that $W = \polydamp$ near the set where $|x| = \sigma$, one must consider both transmitted and reflected geodesics. In our setting reflected geodesics originating in the undamped region remain undamped, which slows decay. Stronger singularities in $W$ correspond to more reflection \cite{DeHoopVasyUhlmann2015, GannotWunsch2018}, so we expect smaller values of $\beta$ to lead to slower decay. (See also \cite{dkk, gw} for examples of such phenomena for scattering resonances) By contrast, in the setting of \cite{LeautaudLerner2017}, where $\sigma =0$, reflected and transmitted geodesics are both equally damped. In that case, smaller values of $\beta$ correspond to larger  time averages of $W$ along geodesics  just because $W$ is then larger, so we expect faster decay. In terms of our estimates below, the effect of geodesics which remain undamped for a long time is reflected in the fact  that  our bounds are weakest for angular momentum modes close to the undamped (vertical) ones; in the notation of Section \ref{s:proof}, this corresponds to $E$ positive but not too large.

\section{Proof of Theorem}\label{s:proof}

By a Fourier transform in time, we may study the associated stationary problem. More precisely,
by Theorem 2.4 of \cite{BorichevTomilov2010}, as formulated in Proposition 2.4 of \cite{AnantharamanLeautaud2014}, the decay \eqref{e:endec} with $\alpha$ given by \eqref{e:ab23} follows from showing that that there are constants $C$ and $q_0$ such that, for any $q \geq q_0$,
\begin{equation}\label{e:stationary}
\| (-\Delta + i q W - q^2)^{-1}\|_{L^2(\Tb) \to L^2(\Tb)} \le C q^{1/(\beta+2)}.
\end{equation}
Expanding in a Fourier series in the $y$ variable  we see that it is enough to show that there are $C$ and $q_0$ such that for any $f \in L^2(\Rb / 2 \pi  \Zb)$, any real $E \leq q^2$ and any $q \geq q_0$, if $u \in H^2( \Rb / 2 \pi \Zb)$ solves
\begin{equation}\label{e:ueq}
 -u'' + i q W u - Eu = f,
\end{equation}
then
\begin{equation}\label{e:2b20}
\int|u|^2 \le C q^{2/(\beta + 2)} \int |f|^2.
\end{equation}
Here, and below, all integrals are over  $\Rb / 2 \pi \Zb$. We will actually obtain a more precise dependence on $E$, namely we will show that there is $E_0>0$ such that
\begin{equation}\label{e:esmall}
\int|u|^2 \le C \int |f|^2, \quad  \text{ when } E \le E_0,
\end{equation}
and
\begin{equation}\label{e:2b2}
\int|u|^2 \le C  E^{-1} q^{2/(\beta + 2)} \int |f|^2, \quad  \text{ when }E \ge E_0.
\end{equation}
The second of these, \eqref{e:2b2}, is our main estimate.

In our proofs we use a version of the Morawetz multiplier method, which we arrange using the energy functional
\begin{equation}\label{e:fdef}
F(x) = |u'(x)|^2 + E|u(x)|^2.
\end{equation}
This method was introduced to prove wave decay for star-shaped obstacle scattering \cite{Morawetz}, and our approach is inspired by that of \cite{CardosoVodev}, as adapted to cylindrical geometry in \cite{ChristiansenDatchev}.

We begin with some easier and essentially well-known estimates. We will often use the elementary fact that if $a,\, b,\, c,\, d,\, e \ge 0$ and $\theta \in [0,1]$, then
\begin{equation}\label{e:elem}
a + b \le c b^{1-\theta} d^{\theta} + e \Longrightarrow a + \theta b \le \theta c^{1/\theta} d + e. 
\end{equation}

\begin{lem}
For any $E \in \Rb, q>0$ and $u,f$ solving \eqref{e:ueq} we have
\begin{equation}\label{e:wu}
 \int W |u|^2 \le q^{-1} \int |fu|.
\end{equation}
Also, for any $\psi  \in C^\infty(\Rb / 2 \pi \Zb)$ which vanishes near $[-\sigma,\sigma]$, there is $C>0$ such that for any $q>0, E \in \Rb$ and $u,f$ solving \eqref{e:ueq} we have
 \begin{equation}\label{e:u'eps}
  \int \psi |u'|^2 \le C (1 + \max(0,E) q^{-1}) \int |fu|.
 \end{equation}
Finally, there are positive constants $E_0$ and $C$ such that for any $q>0, E \le E_0$, and $u,f$ solving \eqref{e:ueq} we have \eqref{e:esmall}.
\end{lem}

\begin{proof}

To prove \eqref{e:wu} we multiply \eqref{e:ueq} by $\bar u $ and take the imaginary part, integrating by parts to see that the first term is real.

To prove \eqref{e:u'eps}, we integrate by parts twice and use \eqref{e:ueq} to write
\begin{equation}\label{e:ibp2}
  \int \psi|u'|^2  = - \re \int \psi' u' \bar u - \re \int \psi u'' \bar u = \frac 12 \int \psi'' |u|^2 + E \int \psi |u|^2 + \re \int \psi f \bar u.
\end{equation}
Now use $|\psi''| + |\psi| \le C W$ and \eqref{e:wu} to conclude.

To prove \eqref{e:esmall}, we multiply \eqref{e:ueq} by $\bar u $ and by a positive function $b \in C^\infty (\Rb / 2 \pi  \Zb)$ to be determined later, integrate, and take the real part to obtain
\[
- \re \int b u'' \bar u - E \int b |u|^2 = \re \int b f \bar u.
\]
Integrating by parts twice (as in \eqref{e:ibp2}), gives
\[
 \int b |u'|^2 + \int \left(-\frac 12b'' - Eb\right)  |u|^2 = \re \int b f \bar u.
\]
Now choose $b$ such that $b''<0$ near $[-\sigma,\sigma]$. Then, as long as $E \leq E_0$ for some $E_0$ sufficiently small, adding a multiple of \eqref{e:wu} gives
\[
 \int  \left( |u'|^2 + |u|^2 \right) \lesssim \int |fu| \le \left(\int|f|^2\right)^{1/2}\left(\int|u|^2\right)^{1/2},
\]
which implies \eqref{e:esmall} by \eqref{e:elem}.
\end{proof}

It remains to show \eqref{e:2b2}. We proceed by proving two lemmas:
\begin{lem}\label{muest}
Let $\delta>0$ be given, and let
\[
 \mu = \mu(x) = \begin{cases}
q^\delta, \qquad  &|x| \in [\sigma, \sigma + q^{-\delta}], \\
1, \qquad &|x| \in [0,\sigma) \cup (\sigma + q^{-\delta}, \pi].
 \end{cases}
\]
Then we have
\begin{equation}
 \int \mu |u'|^2 + E \mu |u|^2 \lesssim \int |f|^2 + q \int W|uu'|.
\end{equation}
\end{lem}
\begin{lem}\label{fuwfu} Let $\delta= \frac{1}{\beta+2}$ and  let
\[
\chi(x)= 
\begin{cases}
	0, \qquad &|x| \in [0, \sigma], \\
	q^\delta(|x| - \sigma), \qquad & |x| \in [\sigma,\sigma+q^{-\delta}], \\
	1, \qquad & |x| \in [\sigma+q^{-\delta},\pi].
\end{cases}
\]
Then we have
\begin{equation}
  \int \mu |u'|^2 + E \mu |u|^2 \lesssim (1+E^{-1}q^{2\delta}) \int |f|^2 +  q^{1/2} \left(\int |fu|\right)^{1/2} \left( \int |W \chi fu|\right)^{1/2}.
\end{equation}
\end{lem}
We then prove \eqref{e:2b2}.

\begin{proof}[Proof of Lemma \ref{muest}] Fix $\tau \in (\sigma,\pi)$ and a continuous and piecewise linear $b$ such that
\[
 b'(x) =
 \begin{cases}
  1, \qquad & |x| \in [0,\sigma),\\
  q^\delta  , \qquad &  |x| \in (\sigma, \sigma + q^{-\delta}),\\
  1, \qquad & |x| \in (\sigma + q^{-\delta}, \tau), \\
  -M, \qquad & |x| \in (\tau, \pi),
 \end{cases}
\]
with $M>0$ chosen such that $b$ is $2\pi$ periodic. We assume $q_0$ is large enough that $\sigma + q^{-\delta}<\tau$ when $q \ge q_0$.

With $F$ as in \eqref{e:fdef}, we have
\[\begin{split}
 (bF)' &= b' |u'|^2 + E b'|u|^2 + 2 b \re u'' \bar u' + 2 E \re u \bar u' \\
 &= b' |u'|^2 + E b'|u|^2 - 2 b \re f \bar u'  + 2 qb \re i W u \bar u'.
\end{split}\]
Using
\[
 \int (bF)' = 0,
\]
gives
\[
 \int b' |u'|^2 + E b' |u|^2 \le 2 \int b| f  u'| + 2 q \int  b W |u u'|.
\]
Add a multiple of \eqref{e:wu} and \eqref{e:u'eps} to both sides, and apply \eqref{e:elem}, to get
the desired statement. 
\end{proof}

\begin{proof}[Proof of Lemma \ref{fuwfu}]
To estimate the last term of Lemma \ref{muest} we use \eqref{e:wu}:
\begin{equation}\label{e:wuu'1}
\left( \int W |uu'|\right)^2 \le \left(\int W |u|^2\right) \left(\int W |u'|^2\right) \lesssim q^{-1} \left(\int |fu|\right) \left( \int \polydamp |u'|^2\right).
\end{equation}
We write
\[
 \int \polydamp|u'|^2  =  \int \polydamp(1-\chi)|u'|^2 +  \int \polydamp\chi|u'|^2.
\]

For the first term use the fact that  $ \polydamp(1-\chi) $ is supported on $[\sigma,\sigma+q^{-\delta}]$ where it obeys
\[
 \polydamp(1-\chi) \le q^{-\delta\beta} = q^{-\delta \beta -\delta} \mu.
\]
To handle the $\polydamp\chi$ term we integrate by parts.
\[
 \int \polydamp\chi |u'|^2 =  - \re \int (\polydamp \chi)' u' \bar u - \re \int \polydamp \chi u'' \bar u.
\]
For the first resulting term we use
\[
 |(\polydamp \chi)'| \lesssim q^\delta W,
\]
and for the other \eqref{e:ueq} and \eqref{e:wu} give
\[
\begin{split}
- \re \int \polydamp \chi u'' \bar u &=  E \int \polydamp \chi |u|^2 + \re \int \polydamp \chi f \bar u \lesssim Eq^{-1} \int|fu| +  \int |W \chi f u|.
\end{split}
 \]
Putting everything into \eqref{e:wuu'1} gives
\[\begin{split}
 \left( \int W |uu'|\right)^2  \lesssim & \ q^{-1-\delta \beta -  \delta} \left(\int |f|^2\right)^{1/2} \left(\int |u|^2\right)^{1/2}   \int \mu |u'|^2 \\
 & + q^{-1+\delta} \left(\int |f|^2\right)^{1/2} \left(\int |u|^2\right)^{1/2}  \int W |uu'| \\
 & + E q^{-2} \left(\int |f|^2\right)\left(\int |u|^2\right) + q^{-1} \left(\int |fu|\right) \left( \int |W \chi fu|\right),
\end{split}\]
which, by \eqref{e:elem}, implies
\[\begin{split}
 \left( \int W |uu'|\right)^2  \lesssim & \ q^{-1-\delta \beta -  \delta}  \left(\int |f|^2\right)^{1/2} \left(\int |u|^2\right)^{1/2}   \int \mu |u'|^2 \\
 & + (E q^{-2} + q^{-2+2\delta}) \left(\int |f|^2\right)\left(\int |u|^2\right) + q^{-1} \left(\int |fu|\right) \left( \int |W \chi fu|\right).
\end{split}\]
Inserting into Lemma \ref{muest} gives
\[\begin{split}
 \int \mu |u'|^2 + E \mu &|u|^2 \lesssim  \int |f|^2 + q^{(1-\delta \beta -  \delta)/2} \left(\int |f|^2\right)^{1/4} \left(\int |u|^2\right)^{1/4}   \left(\int \mu |u'|^2\right)^{1/2} \\
 &+ (E^{1/2} + q^{\delta}) \left(\int |f|^2\right)^{1/2}\left(\int |u|^2\right)^{1/2} + q^{1/2} \left(\int |fu|\right)^{1/2} \left( \int |W \chi fu|\right)^{1/2},
\end{split}\]
and using again \eqref{e:elem} we obtain 
\[
  \int \mu |u'|^2 + E \mu |u|^2 \lesssim (1 + E^{-1}q^{2-2\delta \beta -  2\delta}  + E^{-1}q^{2\delta}) \int |f|^2 +  q^{1/2} \left(\int |fu|\right)^{1/2} \left( \int |W \chi fu|\right)^{1/2}.
\]
We choose $\delta = 1/(\beta + 2)$ to optimize the dependence on $q$, giving Lemma \ref{fuwfu}.
\end{proof}

\begin{proof}[Proof of \eqref{e:2b2}]
Let $\eta_0=\delta$ and let $N \in \mathbb{N}$ to be chosen later and $\eta_j \in (0,\delta)$ with $\eta_{j-1}>\eta_{j}$ for $j = 1, \dots N$ also to be chosen later. By linearity, we may consider separately the $N+3$ cases 
\begin{enumerate}
	\item $|x| \le \sigma$ on $\supp f$, 
	\item $|x| \in [\sigma, \sigma + q^{-\delta}]$ on $\supp f$, 
	\item $|x| \in[\sigma + q^{-\eta_j}, \sigma + q^{-\eta_{j+1}}]$ on $\supp f$, for $j = 0, \dots, N-1$, 
	\item$|x| \ge \sigma + q^{-\eta_N}$ on $\supp f$.
\end{enumerate}

1. In the case that $|x| \le \sigma$ on $\supp f$, the last term in Lemma \ref{fuwfu} vanishes and we have \eqref{e:2b2}.

2. In the case that $|x| \in [\sigma, \sigma + q^{-\delta}]$ on $\supp f$,   we use the fact that $\mu = q^\delta$ there to write
\[
 \int |fu| \le q^{-\delta/2}\left(\int |f|^2\right)^{1/2} \left(\int \mu |u|^2\right)^{1/2},
\]
and, moreover, since $W \le q^{-\beta\delta}$ there, by \eqref{e:wu} we have
\begin{equation}\label{e:wchifu1}\begin{split}
    \int  | W \chi f u| & \le q^{-\beta\delta/2}\int |f W^{1/2} u| \lesssim q^{-1/2 - \beta \delta/2} \left(\int |f|^2\right)^{1/2} \left(\int |fu|\right)^{1/2}  \\
    &\le q^{-1/2 - \beta \delta/2 - \delta/4}  \left(\int |f|^2\right)^{3/4} \left(\int \mu |u|^2 \right)^{1/4}.
  \end{split}
    \end{equation}
Inserted into Lemma \ref{fuwfu}, these give
\[\begin{split}
  \int \mu |u'|^2 + E \mu |u|^2 \lesssim (1+E^{-1} q^{2\delta}) \int |f|^2 + q^{1/4 - \beta \delta/4 - 3\delta/8} \left(\int |f|^2\right)^{5/8} \left(\int \mu |u|^2 \right)^{3/8},
\end{split}\]
which, by \eqref{e:elem}, implies
\[
   \int \mu |u'|^2 + E \mu |u|^2 \lesssim (1+E^{-1}q^{2\delta} + E^{-3/5}q^{2/5 - 2 \beta\delta/5 - 3 \delta/5}) \int |f|^2 ,
\]
which implies \eqref{e:2b2}.

3. In the case that $|x| \in[\sigma + q^{-\eta_j}, \sigma + q^{-\eta_{j+1}}]$, since $W \ge q^{-\eta_j\beta}$ there and by \eqref{e:wu} we have,
\[
 \int |fu| \lesssim q^{\eta_j \beta /2} \int |f W^{1/2} u| \le q^{\eta_j \beta/2 - 1/2}\left(\int |f|^2\right)^{1/2} \left(\int |fu|\right)^{1/2},
\]
or
\[
 \int |fu| \lesssim q^{-1 + \eta_j \beta } \int |f|^2,
\]
which also gives, as in \eqref{e:wchifu1},
\[
\int  |W\chi f u| \lesssim q^{-1/2 - \beta \eta_{j+1}/2}  \left(\int |f|^2\right)^{1/2} \left(\int |fu|\right)^{1/2} \lesssim q^{-1 - \beta \eta_{j+1}/2+ \eta_j \beta/2} \int |f|^2.
\]
Inserting these into Lemma \ref{fuwfu} gives
\[\begin{split}
  \int \mu |u'|^2 + E \mu |u|^2 \lesssim (1+ E^{-1}q^{2\delta}+ q^{-1/2+3 \eta_j \beta/4 - \beta\eta_{j+1}/4}) \int |f|^2. 
\end{split}\]

4. In the case that $|x| \ge \sigma + q^{-\eta_N}$ on $\supp f$ we estimate similarly:
\[
 \int |fu| \lesssim q^{\beta \eta_N/2} \int |f W^{1/2} u| \le q^{ \beta \eta_N /2 - 1/2}\left(\int |f|^2\right)^{1/2} \left(\int |fu|\right)^{1/2},
\]
or
\[
 \int |fu| \lesssim q^{-1+\beta \eta_N} \int |f|^2,
\]
which gives
\[
     \int  |\chi W f u| \lesssim q^{-1/2}  \left(\int |f|^2\right)^{1/2} \left(\int |fu|\right)^{1/2} \lesssim q^{-1+ \beta \eta_N/2} \int |f|^2.
\]
Inserted into Lemma \ref{fuwfu}, these give
\[\begin{split}
  \int \mu |u'|^2 + E \mu |u|^2 \lesssim (1+E^{-1}q^{2\delta} + q^{-1/2 + 3  \beta \eta_N /4}) \int |f|^2.
\end{split}\]
These are optimized when 
\[
\begin{cases} 
3 \eta_j = 4 \eta_{j+1} - \eta_{j+2}, \quad j=0, 1,\ldots, N-2 \\
3 \eta_{N-1} = 4 \eta_{N}.
\end{cases}
\] 
Recalling that $\eta_0 = \delta$ this is solved by 
$$
\eta_k = \delta \frac{3^{N+1}-3^k}{3^{N+1}-1},
$$
this gives \eqref{e:2b2} in all cases 3 and 4 as long as $N$ is chosen large enough that $\beta \le 6(3^{N+1}-1)$.
\end{proof}

\subsection*{Acknowledgments}

The authors are grateful to Jared Wunsch and Matthieu L\'eautaud for helpful comments and suggestions. KD was partially
supported by NSF Grant DMS-1708511. PK was supported in part by the National Science Foundation grant RTG: Analysis on manifolds at Northwestern University.

\bibliographystyle{alpha}
\bibliography{mybib}

\end{document}